\documentclass[12pt]{amsart}

\usepackage[francais,english]{babel}
\usepackage[T1]{fontenc}
\usepackage[latin1]{inputenc}
\usepackage{amsmath,amssymb,amsthm,mathrsfs}
\usepackage[all,2cell]{xy}
\usepackage{graphicx}
\usepackage{enumerate} 
\usepackage{mathrsfs} 
\usepackage{dsfont}
\usepackage{lmodern}

\DeclareMathOperator{\Pro}{Pro}

\DeclareMathOperator{\Fct}{Fct}

\DeclareMathOperator{\RHom}{RHom}

\DeclareMathOperator{\fRHom}{R\mathcal{H}om}

\DeclareMathOperator{\Hom}{Hom}

\DeclareMathOperator{\id}{id}
\DeclareMathOperator{\opp}{op}
\DeclareMathOperator{\Mod}{Mod}

\DeclareMathOperator{\Hn}{H}

\DeclareMathOperator{\gr}{gr}
\DeclareMathOperator{\Rg}{R\Gamma}
\DeclareMathOperator{\dR}{R}
\DeclareMathOperator{\Ob}{Ob}

\begin{document}

\theoremstyle{plain} 
\newtheorem{thm}{Theorem}[section]
\newtheorem{cor}[thm]{Corollary}
\newtheorem{prop}[thm]{Proposition}
\newtheorem{lemme}[thm]{Lemma}
\newtheorem{conj}[thm]{Conjecture}
\newtheorem*{theoetoile}{Theorem} 
\newtheorem*{conjetoile}{Conjecture} 
\newtheorem*{theoetoilefr}{Théorème}
\newtheorem*{propetoilefr}{Proposition}

\theoremstyle{definition} 
\newtheorem{defi}[thm]{Definition}
\newtheorem{example}[thm]{Example}
\newtheorem{examples}[thm]{Examples}
\newtheorem{question}[thm]{Question}
\newtheorem{Rem}[thm]{Remark}
\newtheorem{Notation}[thm]{Notation}

\numberwithin{equation}{section}

\newcommand{\On}[1]{\mathcal{O}_{#1}}
\newcommand{\En}[1]{\mathcal{E}_{#1}}
\newcommand{\Fn}[1]{\mathcal{F}_{#1}} 
\newcommand{\tFn}[1]{\mathcal{\tilde{F}}_{#1}}
\newcommand{\hum}[1]{hom_{\mathcal{A}}({#1})}
\newcommand{\hcl}[2]{#1_0 \lbrack #1_1|#1_2|\ldots|#1_{#2} \rbrack}
\newcommand{\hclp}[3]{#1_0 \lbrack #1_1|#1_2|\ldots|#3|\ldots|#1_{#2} \rbrack}
\newcommand{\catMod}{\mathsf{Mod}}
\newcommand{\Der}{\mathsf{D}}
\newcommand{\Ds}{D_{\mathbb{C}}}
\newcommand{\DG}{\mathsf{D}^{b}_{dg,\mathbb{R}-\mathsf{C}}(\mathbb{C}_X)}
\newcommand{\lI}{[\mspace{-1.5 mu} [}
\newcommand{\rI}{] \mspace{-1.5 mu} ]}
\newcommand{\Ku}[2]{\mathfrak{K}_{#1,#2}}
\newcommand{\iKu}[2]{\mathfrak{K^{-1}}_{#1,#2}}
\newcommand{\Be}{B^{e}}
\newcommand{\op}[1]{#1^{\opp}}
\newcommand{\N}{\mathbb{N}}
\newcommand{\Ab}[1]{#1/\lbrack #1 , #1 \rbrack}
\newcommand{\Du}{\mathbb{D}}
\newcommand{\C}{\mathbb{C}}
\newcommand{\Z}{\mathbb{Z}}
\newcommand{\w}{\omega}
\newcommand{\K}{\mathcal{K}}
\newcommand{\Hoc}{\mathcal{H}\mathcal{H}}
\newcommand{\env}[1]{{\vphantom{#1}}^{e}{#1}}
\newcommand{\eA}{{}^eA}
\newcommand{\eB}{{}^eB}
\newcommand{\eC}{{}^eC}
\newcommand{\cA}{\mathcal{A}} 
\newcommand{\cB}{\mathcal{B}}
\newcommand{\cR}{\mathcal{R}}
\newcommand{\cL}{\mathcal{L}}
\newcommand{\cO}{\mathcal{O}}
\newcommand{\cM}{\mathcal{M}}
\newcommand{\cN}{\mathcal{N}}
\newcommand{\cK}{\mathcal{K}}
\newcommand{\cC}{\mathcal{C}}
\newcommand{\Hper}{\Hn^0_{\textrm{per}}}
\newcommand{\Dper}{\Der_{\mathrm{perf}}}
\newcommand{\Yo}{\textrm{Y}}
\newcommand{\gqcoh}{\mathrm{gqcoh}}
\newcommand{\coh}{\mathrm{coh}}
\newcommand{\cc}{\mathrm{cc}}
\newcommand{\qcc}{\mathrm{qcc}}
\newcommand{\qcoh}{\mathrm{qcoh}}
\newcommand{\obplus}[1][i \in I]{\underset{#1}{\overline{\bigoplus}}}
\newcommand{\Lte}{\mathop{\otimes}\limits^{\rm L}}
\newcommand{\pt}{\textnormal{pt}}
\newcommand{\A}[1][X]{\cA_{{#1}}}
\newcommand{\dA}[1][X]{\cC_{X_{#1}}}
\newcommand{\conv}[1][]{\mathop{\circ}\limits_{#1}}
\newcommand{\sconv}[1][]{\mathop{\ast}\limits_{#1}}
\newcommand{\reim}[1]{\textnormal{R}{#1}_!}
\newcommand{\roim}[1]{\textnormal{R}{#1}_\ast}
\newcommand{\ldetens}{\overset{\mathnormal{L}}{\underline{\boxtimes}}}
\newcommand{\br}{\bigr)}
\newcommand{\bl}{\bigl(}
\newcommand{\sC}{\mathscr{C}}
\newcommand{\ucat}{\mathbf{1}}
\newcommand{\ubtimes}{\underline{\boxtimes}}
\newcommand{\uLte}{\mathop{\underline{\otimes}}\limits^{\rm L}} 
\newcommand{\Lp}{\mathrm{L}p}
\newcommand{\cLte}{\mathop{\overline{\otimes}}\limits^{\rm L}}

\author{Fran\c{c}ois Petit}
\address{Max Planck Institute for Mathematics, Vivatsgasse 7, 53111 Bonn, Germany}
\email{petit@mpim-bonn.mpg.de}
\title[Fourier-Mukai transform in the quantized setting]{Fourier-Mukai transform in the quantized setting}
\begin{abstract} 
We prove that a coherent DQ-kernel induces an equivalence between the derived categories of DQ-modules with coherent cohomology if and only if the graded commutative kernel associated to it induces an equivalence between the derived categories of coherent sheaves. 
\end{abstract}

\maketitle
\section{Introduction}
Fourier-Mukai transform has been extensively studied in algebraic geometry and is still an active area of research (see \cite{Barto} and \cite{Huy}). In the past years, several works have extended to the framework of deformation quantization of complex varieties some important aspects of the theory of integral transforms. In \cite{KS3}, Kashiwara and Schapira have developed the necessary formalism to study integral transforms in the framework of DQ-modules and some classical results have been extended to the quantized setting. In particular, in \cite{PantenBassa}, Ben-Bassat, Block and Pantev have quantized the Poincaré bundle and shown it induces an equivalence between certain derived categories of coherent DQ-modules. 

Our paper grew out of an attempt to understand which properties the integral transforms  associated to the quantization of a coherent kernel would enjoy.

The main result of this paper is Theorem \ref{finalmukai} which states that a coherent DQ-kernel induces an equivalence between the derived categories of DQ-modules with coherent cohomology if and only if the graded commutative kernel associated to it induces an equivalence between the derived categories of coherent sheaves. Whereas the second part of the proof relies on technique of cohomological completion, the first part builds upon the results of \cite{dgaff}. Indeed, as explained in section \ref{adjs} there is a pair of adjoint functors between the categories of qcc objects and the derived category of quasi-coherent sheaves. Both of these functors preserve compact generators. Then, roughly speaking, to show that a certain property of the quantized integral transform implies a similar properties at the commutative level it is sufficient to check that the category of objects satisfying this properties is thick and that this property hold at the quantized level for a compact generator of the triangulated category of qcc objects.

This paper is organized as follow. In the second section we review some material about DQ-modules, cohomological completeness, compactly generated categories, thick subcategories and qcc modules. In the third section, we study integral transforms in the quantized setting. We start by extending the framework of convolutions of kernels of \cite{KS3} to the case of qcc objects and prove that an integral transform of qcc objects preserving compact objects has a coherent kernel (Thm. \ref{thm:kercoh}). Then, we concentrate our attention to the case of integral transforms with coherent kernel. We start by extending to DQ-modules some classical adjunction results and then establish the main theorem of this paper. Finally, in an appendix we show that the cohomological dimension of a certain functor is finite.  

\begin{flushleft}
\textbf{Aknowledgement}: I would like to thank Oren Ben-Bassat, Andrei C\v{a}ld\v{a}raru, Carlo Rossi, Pierre Schapira, Nicol\`o Sibilla, Geordie Williamson for many useful discussions and Damien Calaque and Michel Vaquié for their careful reading of early version of the manuscript and numerous suggestions which have allowed substantial improvements.
\end{flushleft}

%


\section{Some recollections on DQ-modules}\label{adjs}

\subsection{DQ-modules}

We refer the reader to \cite{KS3} for an in-depth study of DQ-modules. Let us briefly fix some notations. 
Let $(X,\cO_X)$ be a smooth complex algebraic variety endowed with DQ-algebroid $\cA_X$. It is possible to define a quotient algebroid stack $\cA_X / \hbar \cA_X$. It comes with a canonical morphism of algebroid stack $\cA_X \to \cA_X / \hbar \cA_X$. On a smooth complex algebraic variety the stack $\cA_X / \hbar \cA_X$ is equivalent to the algebroid stack associated to $\cO_X$. Thus, there is a natural morphism $\cA_X \to \cO_X$ of $\C$-algebroid stacks which induces a functor

\begin{equation*}
\iota_g:\Mod(\cO_X) \to \Mod(\cA_X).
\end{equation*}

The functor $\iota_g$ is exact and fully faithful and induces a functor

\begin{equation*}
\iota_g:\Der(\cO_X) \to \Der(\cA_X).
\end{equation*}

\begin{defi}
We denote by $\gr_\hbar: \Der(\cA_X) \to \Der(\cO_X)$ the left derived functor of the right exact functor $\Mod(\cA_X) \to \Mod(\cO_X)$ given by $\mathcal{M} \mapsto \mathcal{M} / \hbar \mathcal{M}\simeq \cO_X \otimes_{\cA_X} \cM$. For $\mathcal{M} \in \Der(\cA_X)$ we call $\gr_\hbar(\mathcal{M})$ the graded module associated to $\mathcal{M}$. We have
\begin{equation*}
\gr_{\hbar} \mathcal{M} \simeq \cO_X \Lte_{\cA_X} \mathcal{M}.
\end{equation*}
\end{defi}

Finally, we have the following proposition.

\begin{prop}\label{double_adj}
The functor $\gr_\hbar$ and $\iota_g$ define pairs of adjoint functors $(\gr_\hbar,\iota_g)$ and $(\iota_g,\gr_\hbar[-1])$.
\end{prop}

\subsection{Cohomologically complete modules}

We briefly present the notion of cohomologically complete module and state the few results that we need. Again, we refer the reader to \cite[§1.5]{KS3} for a detailed study of this notion. We denote by $\C^\hbar$ the ring of formal power series with coefficient in $\C$. Let $\cR$ be a $\C^\hbar$-algebroid stack without $\hbar$-torsion. We set $\cR_0=\cR/ \hbar \cR$ and $\cR^{loc}=\C^{\hbar,loc} \otimes_{\C^{\hbar}} \mathcal{R}$ where $\C^{\hbar,loc}$ is the field of formal Laurent's series.

\begin{defi}
An object $\cM \in \Der(\cR)$ is cohomologically complete if $\fRHom_{\cR}(\cR^{loc}, \cM) \simeq 0$. We write $\Der_{\cc}(\cR)$ for the full subcategory of $\Der(\cR)$ whose objects are the cohomologically complete modules.
\end{defi}

The category $\Der_{\cc}(\cR)$ is a triangulated subcategory of $\Der(\cR)$.

\begin{prop}[{\cite[Cor. 1.5.9]{KS3}}]\label{ccgr}
Let $\mathcal{M} \in \Der_{\cc}(\mathcal{R})$. If $\gr_{\hbar} \mathcal{M} \simeq 0$, then $\mathcal{M} \simeq 0$.
\end{prop}

\begin{prop}\label{isogr}
Let $f:\mathcal{M} \to \mathcal{N}$ be a morphism of $\Der_{\cc}(\cR)$. If $\gr_{\hbar}(f)$ is an isomorphism then $f$ is an isomorphism. 
\end{prop}

By Proposition 1.5.6 of \cite{KS3}, for any object $\cM$ of $\Der(\cR)$, the object $\fRHom_{\cR}((\cR^{loc}/\cR)[-1],\cM)$ belongs to $\Der_{cc}(\cR)$.

\begin{defi}\label{def:cohocomp}
We denote by $(\cdot)^{\cc}$\index{cc@$(\cdot)^{\cc}$} the functor 
\begin{equation*}
\fRHom_{\cR}((\cR^{loc}/\cR)[-1],\cdot): \Der(\cR) \to \Der(\cR). 
\end{equation*}
We call this functor the functor of cohomological completion.
\end{defi}

The name of functor of cohomological completion is also justified by the fact that $(\cdot)^{\cc} \circ (\cdot)^{\cc} \simeq (\cdot)^{cc}$.

There is a natural transformation
\begin{equation} \label{morcc}
cc:\id \to (\cdot)^{cc}.
\end{equation}
It enjoys the following property.
\begin{prop}[{\cite[Prop. 3.8]{dgaff}}]\label{cciso}
The morphism of functors 
\begin{equation*}
\gr_\hbar(cc):\gr_\hbar \circ \id \to\gr_\hbar \circ (\cdot)^{cc}
\end{equation*}
is an isomorphism in $\Der(\cR_0)$.
\end{prop}

\subsection{Compactly generated categories and thick subcategories}

In this subsection, we review a few facts about compactly generated categories and thick subcategories. These facts play an essential role in the proof of Theorems \ref{thm:kercoh} and \ref{finalmukai}. A classical reference is \cite{Nee_book}. We also refer to \cite[§ 2]{BVdB}.

\begin{defi}
Let $\mathcal{T}$ be a triangulated category. Let $\mathfrak{G}=\lbrace G_i \rbrace_{i \in I}$ be a set of objects of $\mathcal{T}$. One says that $\mathfrak{G}$ generates $\mathcal{T}$ if the following condition is satisfied.
\begin{center}
If $F \in \mathcal{T}$ is such that for every $G_i \in \mathfrak{G}$ and $n \in \Z$ $\Hom_{\mathcal{T}}(G_i[n],F)=0$ then $F \simeq 0$.
\end{center} 
\end{defi}

\begin{defi}
Assume that $\mathcal{T}$ is a cocomplete triangulated category. 

\begin{enumerate} [(i)]
\item An object $L$ in $\mathcal{T}$ is compact if the functor $\Hom_{\mathcal{T}}(L, \cdot)$ commutes with coproducts. We write $\mathcal{T}^{c}$ for the full subcategory of $\mathcal{T}$ whose objects are the compact objects.

\item The category $\mathcal{T}$ is compactly generated if it is generated by a set of compact objects.
\end{enumerate}
\end{defi}

\begin{defi} 
\begin{enumerate}[(i)]
\item A full subcategory of a triangulated category is thick if it is closed under isomorphisms and contains all direct summands of its objects. 

\item The thick envelop $\langle \mathcal{S} \rangle$ of a set of objects $\mathcal{S}$ of a triangulated category $\mathcal{T}$ is the smallest thick triangulated subcategory of $\mathcal{T}$ containing $\mathcal{S}$. 

\item One says that $\mathcal{S}$ classicaly generates $\mathcal{T}$ if its thick envelop is equal to $\mathcal{T}$.
\end{enumerate}
\end{defi}

\begin{thm}[{\cite{Nee_comp} and \cite{Rav}}] \label{Rav_Nee}
Let $\mathcal{T}$ be compactly generated triangulated category. Then a set of compact objects $\mathcal{S}$ of $\mathcal{T}$ classically generates $\mathcal{T}^{c}$ if and only if it generates $\mathcal{T}$.
\end{thm}

The next result is probably well known. We include a proof for the sake of completeness.

\begin{prop} \label{thickiso}
Let $F, \; G: \mathcal{T} \to \mathcal{S}$ be two functors of triangulated categories and $\alpha:F \Rightarrow G$ a natural transformation between them. Then the full subcategory $\mathcal{T}_\alpha$ of $\mathcal{T}$ whose objects are the $X$ such that $\alpha_X:F(X) \to G(X)$ is an isomorphism is a thick subcategory of $\mathcal{T}$.
\end{prop}

\begin{proof}
The category $\mathcal{T}_\alpha$ is triangulated and is closed under isomorphism. Let $X$ be an object of $\mathcal{T}_\alpha$ and $Y$ and $Z$ two objects of $\mathcal{T}$ such that $X \simeq Y \oplus Z$. By definition of the direct sum there is a map $i_Y: Y \to Y \oplus Z$ and a map $p_Y:  Y \oplus Z \to Y$ such that $p_Y \circ i_Y= \id_Y$. Since $\alpha$ is a natural transformation we have the following commutative diagram.
\begin{equation*}
\xymatrix{ F(Y) \ar[r]^-{\alpha_Y} \ar[d]_-{F(i_Y)} & G(Y)  \ar[d]^-{G(i_Y)}\\
F(Y \oplus Z) \ar[r]^-{\sim}_-{\alpha_{Y \oplus Z}} \ar[d]_-{F(p_{Y})} & G(Y \oplus Z)  \ar[d]^-{G(p_{Y})}\\
F(Y) \ar[r]^-{\alpha_Y}  & G(Y).
}
\end{equation*}  
It follows that $F(p_Y) \circ \alpha_{Y \oplus Z}^{-1} \circ G(i_Y)$ is the inverse of $\alpha_Y$. Thus, $Y$ belongs to $\mathcal{T}_\alpha$.
It follows that $\mathcal{T}_\alpha$ is a thick subcategory of $\mathcal{T}$.
\end{proof}

\subsection{Qcc modules}
 We review some facts about qcc modules. They may be considered as a substitute to quasi-coherent sheaves in the quantized setting. For a more detailed study one refers to \cite{dgaff}. In this subsection, $(X, \cO_X)$ is a smooth complex algebraic variety endowed with a DQ-algebroid $\cA_X$. We denote by $\Der_{\qcoh}(\cO_X)$ the derived category of sheaves with quasi-coherent cohomology and by $\Der^b_{\coh}(\cO_X)$ (resp. $\Der^b_{\coh}(\cA_X)$) the derived category of bounded complexes of $\cO_X$-modules (resp. $\cA_X$-modules) with coherent cohomology.  
 
\begin{defi}
An object $\mathcal{M} \in \Der(\cA_X)$ is $\qcc$ if it is cohomologically complete and $\gr_\hbar \cM \in \Der_{\qcoh}(\cO_X)$. The full subcategory of $\Der(\cA_X)$ formed by $\qcc$ modules is denoted by $\Der_{\qcc}(\cA_X)$\index{deriveqcc@$\Der_{\qcc}(\cA_X)$}.
\end{defi}

One easily shows that the category $\Der_{\qcc}(\cA_X)$ is a triangulated subcategory of $\Der(\cA_X)$.

%

\begin{prop}[{\cite[Cor. 3.14]{dgaff}}]
If $\cN \in \Der_{\qcoh}(\cO_X)$, then $\iota_g(\cN) \in \Der_{\qcc}(\cA_X)$.
\end{prop}

The functors $\gr_\hbar$ and $\iota_g$ induce the following functors.

\begin{equation}\label{map:qcccoh}
\xymatrix{ \Der_{\qcc}(\cA_X) \ar@<.4ex>[r]^-{\gr_\hbar} & \Der_{\qcoh}(\cO_X) \ar@<.4ex>[l]^-{\iota_g}.}
\end{equation}

We have the following proposition.

\begin{prop}\label{prop:compres} 
Let $X$ be a smooth complex algebraic variety endowed with a DQ-algebroid $\cA_X$. The functors $\iota_g:\Der_{\qcoh}(\cO_X) \to \Der_{\qcc}(\cA_X)$ and $\gr_\hbar: \Der_{\qcc}(\cA_X) \to \Der_{\qcoh}(\cO_X)$ preserve compact generators.
\end{prop}
%
\begin{proof}
\begin{enumerate}[(i)]
\item We refer to \cite[Cor. 3.15]{dgaff} for the case of the functor $\iota_g$.

\item Let us prove the claim for the functor $\gr_\hbar$. 
Let $\cM \in \Der_{\qcoh}(\cO_X)$ such that $\RHom_{\cO_X}(\gr_\hbar \mathcal{G},\cM)\simeq 0$. Then, we have
\begin{align*}
\RHom_{\cO_X}(\gr_\hbar \mathcal{G},\cM) &\simeq \RHom_{\cA_X}(  \mathcal{G}, \iota_g \cM)\\
&\simeq 0.
\end{align*}
It follows that $\iota_g(\cM) \simeq 0$. Hence, for every $i \in \Z, \, \Hn^i(\iota_g \cM) \simeq 0$. Since $\iota_g:\Mod(\cO_X) \to \Mod(\cA_X)$ is fully faithful and exact, we have $\Hn^i(\cM)\simeq 0$. It follows that $\cM \simeq 0$. Moreover, $\gr_\hbar \mathcal{G}$ is coherent and on a smooth algebraic variety coherent sheaves are compact. 
\end{enumerate}

\end{proof}

\begin{Rem}
We know by \cite{BVdB} that for a complex algebraic variety the category $\Der_{\qcoh}(\cO_X)$ is compactly generated by a single compact object i.e by a perfect complex. As shown in \cite{dgaff}, this implies in particular that $\Der_{\qcc}(\cA_X)$ is compactly generated by a single compact object.
\end{Rem}

\begin{cor}\label{compcohdq}
Let $X$ be a smooth complex algebraic variety endowed with a DQ-algebroid.
\begin{enumerate}[(i)]
\item If $\mathcal{G}$ is a compact generator of $\Der_{\qcoh}(\cO_X)$ then $\gr_\hbar \iota_g \mathcal{G}$ is still a compact generator of $\Der_{\qcoh}(\cO_X)$.
\item One has $\Der^b_{\coh}(\cO_X)= \langle \gr_\hbar \iota_g(\mathcal{G}) \rangle $.
\end{enumerate}
\end{cor}

\begin{proof}
\begin{enumerate}[(i)]
\item Follows immediatly from Proposition \ref{prop:compres}.
\item On a complex smooth algebraic variety the category of compact objects is equivalent to $\Der^b_{\coh}(\cO_X)$. Hence the results follows from Theorem \ref{Rav_Nee}.
\end{enumerate}
\end{proof}

Finally, let us recall the following result from \cite{dgaff}.

\begin{thm}\label{qccompact}
An object $\mathcal{M}$ of $\Der_{\qcc}(\cA_X)$ is compact if and only if $ \mathcal{M} \in \Der^{b}_{\coh}(\cA_X)$ and $\cA_X^{loc} \otimes_{\cA_X} \mathcal{M}=0$.
\end{thm}

\section{Fourier-Mukai functors in the quantized setting} 

The aim of this section is to study integral transforms in the framework of DQ-modules. In the first subsection, we review some results, from \cite{KS3}, concerning the convolution of DQ-kernels. In the second one, we adapt to qcc modules the framework for integral transforms develloped in \cite{KS3}. We prove that an integral transform preserving the compact objects of the qcc has a coherent kernel. In the last subsection, we focus our attention on integral transforms of coherent DQ-modules on projective smooth varieties. We first extend some classical adjunction results and finally prove that a coherent DQ-kernel induced an equivalence between the derived categories of DQ-modules with coherent cohomology if and only if the graded commutative kernel associated to it induces an equivalence between the derived categories of coherent sheaves. 

\vspace{0.3cm}

All along this section we use the following notations.
\begin{Notation}\label{not:www1}
\begin{enumerate}[(i)]
\item If $X$ is a smooth complex variety endowed with a DQ-algebroid $\cA_X$, we denote by $X^a$ the same variety endowed with the opposite DQ-algebroid $\cA_X^{op}$ and we write $\cA_{X^a}$ for this algebroid.

\item
Consider a product of smooth complex varieties $X_1\times X_2 \times X_3$, we write it $X_{123}$. We denote by
$p_i$ the $i$-th projection and by $p_{ij}$ the $(i,j)$-th projection
({\em e.g.,} $p_{13}$ is the projection from 
$X_1\times X_1^a\times X_2$ to $X_1\times X_2$). 
\item
We write $\cA_i$ and $\cA_{ij^a}$
instead of $\cA_{X_i}$ and $\cA_{X_i\times X_j^a}$  and
similarly with other products. 


\item We set $\Du^{\prime}_{\cA_X}=\fRHom_{\cA_X}(\cdot,\cA_X): \Der(\cA_X)^{\opp} \to \Der(\cA_{X^a})$.
\end{enumerate}
\end{Notation}%

\subsection{Convolution of DQ-kernel}
We review some results, from \cite{KS3}, concerning the convolution of DQ-kernels.
\subsubsection{Tensor product and convolution of DQ-kernels} 
The tensor product of DQ-modules is given by

\begin{defi}[{\cite[Def. 3.1.3]{KS3}}]\label{def:DQtens}
Let $\cK_i \in \Der(\cA_{i(i+1)^a})$ $(i=1, 2)$. We set
\begin{equation*}
\cK_1 \uLte_{\cA_2} \cK_2 = p_{12}^{-1} \cK_1 \Lte_{p_{12}^{-1}\cA_{12^a}} \cA_{123} \Lte_{p_{23^a}^{-1} \cA_{23}} p_{23}^{-1} \cK_2 \in \Der(p_{13}^{-1}\cA_{13^a}).
\end{equation*}
\end{defi}

The composition of kernels is given by

\begin{defi}\label{def:DQconv}
Let $\cK_i \in \Der(\cA_{i(i+1)^a})$ $(i=1, 2)$. We set
\begin{align*}
\cK_1 \underset{2}{\ast} \cK_2 =& \dR p_{13 \ast} (\cK_1 \uLte_{\cA_2} \cK_2) \in \Der(\cA_{13^a})\\
\cK_1 \underset{2}{\circ} \cK_2 =& \dR p_{13!} (\cK_1 \uLte_{\cA_2} \cK_2) \in \Der(\cA_{13^a}).
\end{align*}
\end{defi}

\subsubsection{Finiteness and duality for DQ-modules}

The following result is a special case of Theorem 3.2.1 of \cite{KS3}. 
%
%
\begin{thm} \label{coherence}
Let $X_i$ $(i=1, \; 2, \; 3)$ be a smooth complex variety. For $i=1 \, ,2$, consider the product $X_i \times X_{i+1}$ and $\cK_i \in \Der_{\coh}^b(\cA_{i (i+1)^a})$. Assume that $X_2$ is proper. Then the object $\cK_1 \underset{2}{\circ} \cK_2$ belongs to $\Der^b_{\coh}(\cA_{13^a})$.
\end{thm}

Let $X_i$ $(i=1, \, 2, \, 3)$ be a smooth projective complex variety endowed with the Zariski topology and let $\cA_i$ be a DQ-algebroid on $X_i$. 
We recall some duality results for DQ-modules from \cite[Chap. 3]{KS3}. First, we need the following result.
\begin{prop} [{\cite[p. 93]{KS3}}]
Let $\cK_i \in \Der^b(\cA_{i(i+1)^a})$ $(i=1 \, ,2)$ and let $\mathcal{L}$ be a bi-invertible $\cA_2 \otimes \cA_{2^a}$-module. Then, there is a natural isomorphism
\begin{equation*}
(\cK_1 \underset{2}{\circ} \mathcal{L})\underset{2}{\circ} \cK_2 \simeq \cK_1 \underset{2}{\circ} (\mathcal{L}\underset{2}{\circ} \cK_2).
\end{equation*}
\end{prop}


We denote by $\w_i$ the dualizing complexe for $\cA_i$. It is a bi-invertible $(\cA_i \otimes \cA_{i^a})$-module. Since the category of bi-invertible $(\cA_i \otimes \cA_{i^a})$-modules is equivalent to the category of coherent $\cA_{ii^a}$-modules simple along the diagonal, we will regard $\w_i$ as an $\cA_{ii^a}$-module supported by the diagonal and we will still denote it by $\w_i$.

\begin{thm}[{\cite[Theorem 3.3.3]{KS3}}] \label{dualdqserre}
Let $\cK_i \in \Der_{\coh}^b(\cA_{i (i+1)^a})$ $(i=1, \, 2)$. 
There is a natural isomorphism in $\Der^b_{\coh}(\cA_{1^a3})$
\begin{equation*}
(\Du^{\prime}_{\cA_{12^a}}\cK_1) \underset{2^a}{\circ} \w_{2^a} \underset{2^a}{\circ} (\Du^{\prime}_{\cA_{23^a}}\cK_2) \stackrel{\sim}{\to} \Du^{\prime}_{\cA_{13^a}}(\cK_1 \underset{2}{\circ} \cK_2).
\end{equation*}
\end{thm}

\subsection{Integral transforms for qcc modules}

In this section, we adapt to qcc objects the framework of convolutions of kernels of \cite{KS3}. In view of Definitions \ref{def:DQtens} and \ref{def:DQconv}, it is easy, using the functor of cohomological completion (see Definition \ref{def:cohocomp}), to define a tensor product and a composition for cohomologically complete modules.

\begin{defi}
Let $\cK_i \in \Der_{\cc}(\cA_{i(i+1)^a})$ $(i=1, 2)$. We set
\begin{align*}
\cK_1 \cLte_{\cA_2} \cK_2 =& (\cK_1 \uLte_{\cA_2} \cK_2)^{\cc} \in \Der_{cc}(p_{13}^{-1}\cA_{13^a}),\\
\cK_1 \underset{2}{\overline{\ast}} \cK_2 =& \dR p_{13 \ast} (\cK_1 \cLte_{\cA_2} \cK_2) \in \Der_{cc}(\cA_{13^a}),\\
\cK_1 \underset{2}{\overline{\circ}} \cK_2 =& \dR p_{13!} (\cK_1 \cLte_{\cA_2} \cK_2) \in \Der_{cc}(\cA_{13^a}).
\end{align*}
\end{defi}

\begin{Rem}
If $\cK_1 \in \Der^b_{\coh}(\cA_{12^a})$ and $\cK_2 \in \Der^b_{\coh}(\cA_{23^a})$, then
\begin{align*}
\cK_1 \cLte_{\cA_2} \cK_2 & \simeq \cK_1 \uLte_{\cA_2} \cK_2,\\
\cK_1 \underset{2}{\overline{\ast}} \cK_2 & \simeq \cK_1 \underset{2}{\ast} \cK_2,\\
\cK_1 \underset{2}{\overline{\circ}} \cK_2 & \simeq \cK_1 \underset{2}{\circ} \cK_2.
\end{align*}

\end{Rem}

\begin{lemme}\label{lemme:inversioncc}
Let $\cK_i \in \Der_{\cc}(\cA_{i(i+1)^a})$ $(i=1, 2)$. 
\begin{align*}
\cK_1 \underset{2}{\overline{\ast}} \cK_2 \simeq & (\cK_1 \underset{2}{\ast} \cK_2)^{\cc},\\
\cK_1 \underset{2}{\overline{\circ}} \cK_2 \simeq & (\cK_1 \underset{2}{\circ} \cK_2)^{\cc}.
\end{align*}
\end{lemme}

\begin{proof}
Using morphism (\ref{morcc}), we get a map
\begin{equation*}
\cK_1 \uLte_{\cA_2} \cK_2 \to \cK_1 \cLte_{\cA_2} \cK_2.
\end{equation*}
It induces a morphism
\begin{equation*}
(\dR p_{13 \ast} (\cK_1 \uLte_{\cA_2} \cK_2))^{cc} \to (\dR p_{13 \ast} (\cK_1 \cLte_{\cA_2} \cK_2))^{cc}.
\end{equation*}
By Proposition 1.5.12 of \cite{KS3} the direct image of a cohomologically complete module is cohomologically complete. Then,
\begin{equation*}
(\dR p_{13 \ast} (\cK_1 \cLte_{\cA_2} \cK_2))^{cc} \simeq \dR p_{13 \ast} (\cK_1 \cLte_{\cA_2} \cK_2).
\end{equation*}
This gives us a map
\begin{equation}\label{mor:comcc}
 (\cK_1 \underset{2}{\ast} \cK_2)^{cc} \to  \cK_1 \underset{2}{\overline{\ast}} \cK_2.
\end{equation}

Using the fact that the functor $\gr_\hbar$ commutes with direct image
and Proposition \ref{cciso}, we get the following commutative diagram.

\begin{equation*}
\xymatrix{\gr_\hbar ((\cK_1 \underset{2}{\ast} \cK_2)^{cc}) \ar[r]& \gr_\hbar( \cK_1 \underset{2}{\overline{\ast}} \cK_2) \\
\gr_\hbar (\cK_1 \underset{2}{\ast} \cK_2) \ar[u]^-{\gr_\hbar(cc)}_-{\wr} \ar[r] & \gr_\hbar( \cK_1 \underset{2}{\overline{\ast}} \cK_2) \ar@{=}[u]\\
\dR p_{13 \ast} \gr_\hbar (\cK_1 \uLte_{\cA_2} \cK_2) \ar[u]_-{\wr} \ar[r]^-{\sim}_-{\gr_\hbar(cc)} & (\dR p_{13 \ast} \gr_\hbar (\cK_1 \cLte_{\cA_2} \cK_2)) \ar[u]^-{\wr}. \\
}
\end{equation*}


It follows that the morphism $\gr_\hbar ((\cK_1 \underset{2}{\ast} \cK_2)^{cc}) \to \gr_\hbar( \cK_1 \underset{2}{\overline{\ast}} \cK_2)$ is an isomorphism. Applying Proposition \ref{isogr}, we obtain that the morphism (\ref{mor:comcc}) is an isomorphism.

The second formula is proved similarly.
\end{proof}

From now on all the varieties considered are smooth complex algebraic varieties endowed with the Zariski topology

\begin{cor}
Let $\cK_i \in \Der_{\qcc}(\cA_{i(i+1)^a})$ $(i=1, 2)$. The kernel $
\cK_1 \underset{2}{\overline{\ast}} \cK_2$ is an object of $\Der_{\qcc}(\cA_{13^a})$.
\end{cor}

\begin{proof}
This follows from Lemma \ref{lemme:inversioncc} and \cite[Prop 3.1.4]{KS3} which says that the functor $\gr_\hbar$ commutes with the compostion of DQ-kernel (see Definition \ref{def:DQconv}).
\end{proof}

Let $\cK \in \Der_{\qcc}(\cA_{12^a})$. The above corollary implies that the functor (\ref{MukaiDQqcc}) is well-defined.

\begin{equation} \label{MukaiDQqcc}
\scalebox{0.96}{$\Phi_\cK: \Der_{\qcc}(\cA_2) \to \Der_{\qcc}(\cA_1), \quad  \cM \mapsto \cK \underset{2}{\overline{\ast}} \cM = Rp_{1\ast}(\cK \cLte_{p_2^{-1}\cA_2} p_2^{-1} \cM).$} 
\end{equation}

Before proving Theorem \ref{thm:kercoh}, we need to establish the following result.

\begin{prop}\label{prop:boundingcomplex}
Let $\cM \in \Der_{cc}(\cA_X)$. If $\gr_\hbar \cM \in \Der^b(\cO_X)$ then $\cM \in \Der^b_{cc}(\cA_X)$.
\end{prop}

\begin{proof}
Let $\cM \in \Der_{cc}(\cA_X)$ such that $\gr_\hbar \cM \in \Der^b(\cO_X)$. It follows immediately from \cite[Prop. 1.5.8]{KS3}, that $\cM \in \Der^{+}(\cA_X)$. Then to establish that  $\cM \in \Der^{b}(\cA_X)$, it is sufficient to prove that there exists a number $q$ such that $\tau^{\geq q} \cM \in \Der^{b}(\cA_X)$. For that purpose, we essentially follow the proof of Proposition 1.5.8 of \cite{KS3}. Since $\gr_\hbar \cM \in \Der^b_{\coh}(\cO_X)$, there exists $p \in \Z$ such that for every $i \geq p$, $\Hn^i(\gr_\hbar \cM)=0$. We deduce from the exact sequence $\Hn^i(\gr_\hbar \cM) \to \Hn^{i+1}(\cM) \stackrel{\hbar}{\to} \Hn^{i+1}(\cM) \to \Hn^{i+1}(\gr_\hbar \cM)$ that $\Hn^i(\cM) \stackrel{\hbar}{\to} \Hn^i(\cM)$ is an isomorphism for $i>p$. Thus, $\tau^{\geq p+1} \cM \in \Der(\cA_X^{loc})$ which means that $\fRHom_{\cA_X}(\cA_X^{loc},\tau^{\geq p+1} \cM)\simeq \tau^{\geq p+1} \cM$. Applying $\fRHom_{\cA_X}(\cA_X^{loc}, \cdot)$ to the distinguished triangle 
\begin{equation*}
\scalebox{0.9}{$\tau^{\leq p} \cM \to \cM \to \tau^{\geq p+1} \cM \stackrel{+1}{\to}$},
\end{equation*}
we get the distinguished triangle
\begin{equation*}
\scalebox{0.9}{$\fRHom_{\cA_X}(\cA_X^{loc},\tau^{\leq p} \cM) \to \fRHom_{\cA_X}(\cA_X^{loc}, \cM) \to \fRHom_{\cA_X}(\cA_X^{loc},\tau^{\geq p+1} \cM) \stackrel{+1}{\to}.$}
\end{equation*}
The module $\cM$ is cohomologically complete. Hence, we have the isomorphism $\fRHom_{\cA_X}(\cA_X^{loc},\cM)\simeq 0$. It follows that 
\begin{equation*}
\tau^{\geq p+1}\cM \simeq \fRHom_{\cA_X}(\cA_X^{loc},\tau^{\leq p} \cM)[1].
\end{equation*}
Corollary \ref{cor:wayout} implies that $\fRHom_{\cA_X}(\cA_X^{loc},\tau^{\leq p} \cM)[1] \in \Der^{-}(\cA_X)$. Then $\tau^{\geq p+1}\cM \in \Der^{+}(\cA_X) \cap \Der^{-}(\cA_X)=\Der^{b}(\cA_X)$. Thus, $\cM \in \Der^{b}(\cA_X)$.
\end{proof}

We now restrict our attention to the case of smooth proper algebraic varieties. The next result is inspired by \cite[Thm. 8.15]{mordg}. Recall that the objects of $\Der^b_{\coh}(\cA_X)$ are not necessarily compact in $\Der_{\qcc}(\cA_X)$ (see Theorem \ref{qccompact}).

\begin{thm}\label{thm:kercoh}
Let $X_1$ (resp. $X_2$) be a smooth complex algebraic variety endowed with a DQ-algebroid $\cA_1$ (resp. $\cA_2$). Let $\cK \in \Der_{\qcc}(\cA_{12^a})$. Assume that the functor $\Phi_\cK: \Der_{\qcc}(\cA_2) \to \Der_{\qcc}(\cA_1)$ preserves compact objects. Then, $\cK$ belongs to $\Der^{b}_{\coh}(\cA_{12^a})$.
\end{thm}

\begin{proof}
The kernel $\gr_\hbar\cK$ induces an integral transform 
\begin{equation*}
\Phi_{\gr_\hbar \cK } :\Der_{\qcoh}(\cO_2) \to \Der_{\qcoh}(\cO_1). 
\end{equation*}

Let $\mathcal{G}$ be a compact generator of $\Der_{\qcoh}(\cO_2)$. Then, by Proposition \ref{prop:compres} $\iota_g(\mathcal{G})$ is a compact generator of $\Der_{\qcc}(\cA_2)$. By hypothesis, $\Phi_\cK (\iota_g(\mathcal{G}))$ is a compact object of $\Der_{\qcc}(\cA_1)$. It follows that the object $\Phi_{\gr_\hbar \cK} (\gr_\hbar \iota_g(\mathcal{G}))$ belongs to $\Der_{\coh}^b(\cO_1)$ and thus is a compact object of $\Der_{\qcoh}(\cO_1)$.

Let $\mathcal{T}$ be the full subcategory of $\Der_{\coh}^b(\cO_2)$ such that $\Ob(\mathcal{T})=\lbrace \cM \in \Der^b_{\coh}(\cO_2) | \Phi_{\gr_\hbar \cK} (\cM) \in \Der_{\coh}^b(\cO_1)\rbrace$. The category $\mathcal{T}$ is a thick subcategory of $\Der_{\coh}^b(\cO_2)$ containing $\gr_\hbar \iota_g (\mathcal{G})$. By Corollary \ref{compcohdq}, $\Der_{\coh}^{b}(\cO_2)$ is the thick envelop of $\gr_\hbar \iota_g(\mathcal{G})$. Thus, $\mathcal{T}=\Der^b_{\coh}(\cO_2)$. It follows that the image of an object of $\Der_{\coh}^{b}(\cO_2)$ by $\Phi_{\gr_\hbar \cK }$ is an object of $\Der_{\coh}^b(\cO_1)$. Applying Theorem 8.15 of \cite{mordg}, we get that $\gr_\hbar \cK$ is an object of $\Der^b_{\coh}(\cO_{12})$. Applying Proposition \ref{prop:boundingcomplex}, we get that $\cK \in \Der^b(\cA_{12^a})$. Now, Theorem 1.6.4 of \cite{KS3} implies that $\cK \in \Der^{b}_{\coh}(\cA_{12^a})$.
\end{proof}

\subsection{Integral transforms of coherent DQ-modules}

In this section we study integral transforms of coherent DQ-modules. 
%
%
%
%
%
Recall that all the varieties considered are smooth complex projective varieties endowed with the Zariski topology.

\vspace{0.3cm}

Let $\cK \in \Der^b_{\coh}(\cA_{12^a})$. Theorem \ref{coherence} implies that the functor (\ref{MukaiDQ}) is well-defined.

\begin{equation} \label{MukaiDQ}
\scalebox{0.96}{$\Phi_\cK: \Der^b_{\coh}(\cA_2) \to \Der^b_{\coh}(\cA_1),\; \cM \mapsto \cK \underset{2}{\circ} \cM = Rp_{1\ast}(\cK \Lte_{p_2^{-1}\cA_2} p_2^{-1} \cM).$} 
\end{equation}

\begin{prop}
Let $\cK_1 \in \Der^b_{\coh}(\cA_{12^a})$ and $\cK_2 \in \Der^b_{\coh}(\cA_{23^a})$. The composition
\begin{align*}
\Der^b_{\coh}(\cA_3)\stackrel{\Phi_{\cK_2}}{\rightarrow}\Der^b_{\coh}(\cA_2) \stackrel{\Phi_{\cK_1}}{\rightarrow}\Der^b_{\coh}(\cA_1)
\end{align*}
is isomorphic to $\Phi_{\cK_1 \underset{2}{\circ} \cK_2}:\Der^b_{\coh}(\cA_3) \to \Der^b_{\coh}(\cA_1)$.
\end{prop}

\begin{proof}
It is a direct consequence of Proposition 3.2.4 of \cite{KS3}.
\end{proof}

We extend to DQ-modules some classical adjunctions results. They are usually established using Grothendieck duality which does not seem possible to do here. Our proof relies on Theorem \ref{dualdqserre}.

\begin{defi}
For any object $\cK \in \Der^b_{\coh}(\cA_{12^a})$, we set
\begin{align*}
\cK_R= \Du^{\prime}_{\cA_{12^a}}(\cK) \underset{2^a}{\circ} \w_{2^a} && \cK_L= \w_{1^a} \underset{1^a}{\circ} \Du^{\prime}_{\cA_{12^a}}(\cK)  
\end{align*}
objects of $\Der_{\coh}^b(\cA_{1^a2})$.
\end{defi}

\begin{prop}\label{adjoint_mukai}
Let $\Phi_\cK:\Der^b_{\coh}(\cA_2) \to \Der^b_{\coh}(\cA_1)$ be the Fourier-Mukai functor associated to $\cK$ and $\Phi_{\cK_R}:\Der^b_{\coh}(\cA_1) \to \Der^b_{\coh}(\cA_2)$ (resp. $\Phi_{\cK_L}:\Der^b_{\coh}(\cA_1) \to \Der^b_{\coh}(\cA_2)$) the Fourier-Mukai functor associated to $\cK_R$ (resp. $\cK_L$). Then $\Phi_{\cK_R}$ (resp. $\Phi_{\cK_L}$) is right (resp. left) adjoint to $\Phi_\cK$.
\end{prop}

\begin{proof}
We have
\begin{equation*}
\RHom_{\cA_1}(\cK \underset{2}{\circ} \cM, \cN) \simeq \Rg(X_1,\fRHom_{\cA_1}(\cK \underset{2}{\circ} \cM, \cN)).
\end{equation*}

Applying Theorem \ref{dualdqserre} and the projection formula, we get

\begin{align*}
\fRHom_{\cA_1}(\cK \underset{2}{\circ} \cM, \cN) & \simeq \fRHom_{\cA_1}(\cK \underset{2}{\circ} \cM, \cA_1) \Lte_{\cA_1} \cN\\
&\simeq (\Du^{\prime}_{\cA_{12^a}}(\cK)\underset{2^a}{\circ} \w_{2^a} \underset{2^a}{\circ} \Du^{\prime}_{\cA_2}(\cM)) \Lte_{\cA_1} \cN\\
&\simeq (\cK_R \underset{2^a}{\circ} \Du^{\prime}_{\cA_2}(\cM)) \Lte_{\cA_1} \cN\\
&\simeq \dR p_{1 \ast} ( \cK_R \Lte_{p_2^{-1} \cA_{2^a}} p_2^{-1} \Du^{\prime}_{\cA_2}(\cM)) \Lte_{\cA_1} \cN\\
&\simeq \dR p_{1 \ast} ( \cK_R \Lte_{p_2^{-1} \cA_{2^a}} p_2^{-1} \Du^{\prime}_{\cA_2}(\cM) \Lte_{p_1^{-1}\cA_1} p_1^{-1} \cN).
\end{align*}

Taking the global section and applying again the projection formula, we get
\begin{align*}
\Rg(X_1,\fRHom_{\cA_1}(\cK \underset{2}{\circ} \cM, \cN)) \simeq &\\ & \hspace{-2.8cm}\Rg(X_1, \dR p_{1 \ast} ( \cK_R \Lte_{p_2^{-1} \cA_{2^a}} p_2^{-1} \Du^{\prime}_{\cA_2}(\cM)) \Lte_{p_1^{-1}\cA_1} p_1^{-1} \cN)\\
&\hspace{-2.8cm} \simeq \Rg(X_1 \times X_2,(\cK_R \Lte_{p_2^{-1} \cA_{2^a}} p_2^{-1} \Du^{\prime}_{\cA_2}(\cM)) \Lte_{p_1^{-1}\cA_1} p_1^{-1} \cN)\\
& \hspace{-2.8cm} \simeq \Rg(X_2, \dR p_{2 \ast} ( (\cK_R \Lte_{p_2^{-1} \cA_{2^a}} p_2^{-1} \Du^{\prime}_{\cA_2}(\cM)) \Lte_{p_1^{-1}\cA_1} p_1^{-1} \cN))\\
&\hspace{-2.8cm} \simeq \Rg(X_2,\Du^{\prime}_{\cA_2}(\cM) \Lte_{\cA_2} (\cK_R \underset{1}{\circ} \cN))\\
& \hspace{-2.8cm} \simeq \RHom_{\cA_2}(\cM,\cK_R \underset{1}{\circ} \cN).
\end{align*}
Thus, $\RHom_{\cA_1}(\cK \underset{2}{\circ} \cM, \cN)\simeq \RHom_{\cA_2}(\cM,\cK_R \underset{1}{\circ} \cN)$ which proves the claim. The proof is similar for $\cK_L$.
\end{proof}

Finally, we have the following theorem.

\begin{thm}\label{finalmukai}
Let $X_1$ (resp. $X_2$) be a smooth complex projective variety endowed with a DQ-algebroid $\cA_1$ (resp. $\cA_2$). Let $\cK \in \Der^b_{\coh}(\cA_{12^a})$. The following conditions are equivalent
\begin{enumerate}[(i)]
\item The functor $\Phi_\cK: \Der^b_{\coh}(\cA_2) \to \Der^b_{\coh}(\cA_1)$ is fully faithful  (resp. an equivalence of triangulated categories).

\item The functor $\Phi_{\gr_\hbar \cK}:\Der^b_{\coh}(\cO_2) \to \Der^b_{\coh}(\cO_1)$ is fully faithful (resp. an equivalence of triangulated categories).
\end{enumerate}
\end{thm}

\begin{proof}
We recall the following fact. Let $F$ and $G$ be two functors and assume that $F$ is right adjoint to $G$. Then, there are two natural morphisms
\begin{align}
G \circ F \to \id \label{eq:1f}\\
\id \to F \circ G \label{eq:2f}.
\end{align}
The morphism (\ref{eq:1f}) (resp. (\ref{eq:2f})) is an isomorphism if and only if $F$ (resp. G) is fully faithfull. The morphisms (\ref{eq:1f}) and (\ref{eq:2f}) are isomorphisms if and only if $F$ and $G$ are equivalences.

\begin{enumerate}
\item $(i) \Rightarrow (ii)$.
Proposition \ref{adjoint_mukai} is also true for $\cO$-modules since the proof works in the commutative case without any changes. Moreover, the functor $\gr_\hbar$ commutes with the composition of kernels. Hence, we have $\gr_\hbar (\cK_R) \simeq (\gr_\hbar \cK)_R$.
Therefore, the functor $\Phi_{\gr_\hbar \cK_R}$ is a right adjoint of the functor $\Phi_{\gr_\hbar \cK}$. Thus, there are morphisms of functors
\begin{align}
\Phi_{\gr_\hbar \cK} \circ \Phi_{\gr_\hbar \cK_R} \to \id, \label{isoeqi1}\\
\id \to \Phi_{\gr_\hbar \cK_R} \circ \Phi_{\gr_\hbar \cK}. \label{isoeqi2}
\end{align}
Set $\Phi_{\cL}=\Phi_{\gr_\hbar \cK_R} \circ \Phi_{\gr_\hbar \cK}$.
Let $\mathcal{T}_2$ be the full subcategory of $\Der^b_{\coh}(\cO_2)$ whose objects are the $\cM \in \Der^b_{\coh}(\cO_2)$ such that
\begin{equation*}
\cM \to \Phi_{\cL}(\cM)
\end{equation*}
is an isomorphism. It follows from Proposition \ref{thickiso} that $\mathcal{T}_2$ is a thick subcategory of $\Der_{\coh}^b(\cO_2)$.

Let $\mathcal{G}$ be a compact generator of $\Der_{\qcoh}(\cO_2)$. By Corollary \ref{compcohdq}, $\Der_{\coh}^b(\cO_2)=\langle \mathcal{G} \rangle$. Since $\Phi_{\cK}$ is a fully faithful we have the isomorphism
\begin{equation*}
\iota_g(\mathcal{G}) \stackrel{\sim}{\to} \Phi_{\cK_R} \circ \Phi_{\cK}(\iota_g(\mathcal{G})).
\end{equation*}
Applying the functor $\gr_\hbar$, we get that $\gr_\hbar \iota_g(\mathcal{G})$ belongs to $\mathcal{T}_2$ and by Corollary \ref{compcohdq}, $\gr_\hbar \iota_g(\mathcal{G})$ is a classical generator of 
$\Der^b_{\coh}(\cO_2)$. Hence, $\mathcal{T}_2=\Der_{\coh}^b(\cO_2)$.
Thus, the morphism (\ref{isoeqi2}) is an isomorphism of functors. A similar argument shows that if $\Phi_{\gr_\hbar \cK}$ is an equivalence the morphism (\ref{isoeqi1}) is also an isomorphism which proves the claim.

\item $(ii) \Rightarrow (i)$. 

Since $\Phi_\cK$ and $\Phi_{\cK_R}$ are adjoint functors we have natural morphisms of functors
\begin{align*}
\Phi_{\cK} \circ \Phi_{\cK_R} \to \id, \\
\id \to \Phi_{\cK_R} \circ \Phi_{\cK}.
\end{align*}

If $\cM \in \Der^b_{\coh}(\cA_2)$, then we have

\begin{equation} \label{buttraite}
\cM \to \Phi_{\cK_R} \circ \Phi_{\cK}(\cM).
\end{equation}

Applying the functor $\gr_\hbar$, we get

\begin{equation}\label{equicom}
\gr_\hbar \cM \to \Phi_{\gr_\hbar \cK_R} \circ \Phi_{\gr_\hbar \cK}(\gr_\hbar \cM).
\end{equation}

If $\Phi_{\gr_\hbar \cK}$ is fully faithful, then the morphism (\ref{equicom}) is an isomorphism. The objects $\Phi_{\cK_R} \circ \Phi_{\cK}(\cM)$ and $\cM$ are cohomologically complete since they belongs to $\Der^b_{\coh}(\cA_2)$.
Thus the morphism (\ref{buttraite}) is an isomorphism that is to say
\begin{equation*}
\id \stackrel{\sim}{\to} \Phi_{\cK_R} \circ \Phi_{\cK}.
\end{equation*}%
It follows that $\Phi_\cK$ is fully faithful.

Similarly, one shows that if $\Phi_{\gr_\hbar \cK}$ is an equivalence then in addition
\begin{equation*}
\Phi_{\cK} \circ \Phi_{\cK_R} \stackrel{\sim}{\to} \id.
\end{equation*}
It follows that $\Phi_{\cK}$ is an equivalence.
\end{enumerate}
\end{proof}

\begin{Rem}
The implication $(ii) \Rightarrow (i)$ of Theorem \ref{finalmukai} and Proposition \ref{adjoint_mukai} still hold if one replaces smooth projective varieties by complex compact manifolds. This result implies immediatly that the quantization of the Poincaré bundle constructed in \cite{PantenBassa} induces an equivalence.
\end{Rem}

\section{Appendix}

In this appendix we show that the cohomological dimension of the functor $\fRHom_{\C^\hbar_X}(\C^{\hbar,loc}_X, \cM)$ is finite.
We refer to \cite{categories_and_sheaves} for a detailed account of pro-objects. Recall that to an abelian category $\mathcal{C}$ one associates the abelian category $\Pro(\mathcal{C})$ of its pro-objects. Then, there is a natural fully faithful functor $i_\mathcal{C}: \mathcal{C} \to \Pro(\mathcal{C})$. The functor $i_\mathcal{C}$ is exact. For any small filtrant category $I$ the functor $``\varprojlim":\Fct(I^{\opp}, \mathcal{C}) \to \Pro(\mathcal{C})$ is exact. If $\mathcal{C}$ admits small projective limits the functor $i_\mathcal{C}$ admits a right adjoint denoted $\pi$.

\begin{equation*}
\pi:\Pro(\mathcal{C}) \to \mathcal{C} , \; "\varprojlim_i" X_i \mapsto \varprojlim_i X_i.
\end{equation*}
 
If $\mathcal{C}$ is a Grothendieck category, then $\pi$ has a right derived functor $\dR \pi:\Der(\Pro(\mathcal{C})) \to \Der(\mathcal{C})$.

Let us recall Lemma 1.5.11 of \cite{KS3}.

\begin{lemme}\label{lemme:procomp}
Let $\cM \in \Der(\C^\hbar)$. Then , we have
\begin{equation}
\dR \pi (( `` \varprojlim_{n}" \C^\hbar_X \hbar^n) \Lte_{\C^\hbar_X} \cM) \simeq \fRHom_{\C^\hbar_X}(\C_X^{\hbar,loc}, \cM).
\end{equation}
\end{lemme}

\begin{prop}
The functor $\fRHom_{\C^\hbar_X}(\C^{\hbar,loc}_X , \cdot)$ has finite cohomological dimension.
\end{prop}

\begin{proof}
Let $\cM \in \Mod(\C^\hbar_X)$. By Lemma \ref{lemme:procomp} we have

\begin{equation}
\fRHom_{\C^\hbar_X}(\C_X^{\hbar,loc}, \cM) \simeq \dR \pi (( `` \varprojlim_{n}" \C^\hbar_X \hbar^n) \Lte_{\C^\hbar_X} \cM).
\end{equation}

 Then by Proposition 6.1.9 of \cite{categories_and_sheaves} adapted to the case of pro-objects, we have

\begin{equation*}
\dR \pi (( `` \varprojlim_{n}" \C^\hbar_X \hbar^n) \Lte_{\C^\hbar_X} \cM)\simeq \dR \pi ( `` \varprojlim_{n}" (\C^\hbar_X \hbar^n \Lte_{\C^\hbar_X} \cM)).
\end{equation*} 

It follows from Corollary 13.3.16 from \cite{categories_and_sheaves} that 
\begin{equation*}
\forall i>1, \; \dR^i \pi ( `` \varprojlim_{n}" (\C^\hbar_X \hbar^n \Lte_{\C^\hbar_X} \cM)) \simeq 0 
\end{equation*} 
which proves the claim.
\end{proof}

\begin{prop}
The functor $\fRHom_{\C^\hbar_X}(\C_X^{\hbar,loc},\cdot):\Der(\C^\hbar_X) \to \Der(\C^\hbar_X)$ is such that $\fRHom_{\C^\hbar_X}(\C_X^{\hbar,loc},\Der^{-}(\C^\hbar_X))\subset \Der^{-}(\C^\hbar_X)$.
\end{prop}
\begin{proof}
This follows immediately from Example 1 of \cite[Ch. I §7.]{resdu}. 
\end{proof}

\begin{cor}\label{cor:wayout}
Let $X$ be a smooth complex (algebraic or analytic) variety endowed with a DQ-algebroid $\cA_X$. The functor $\fRHom_{\cA_X}(\cA_X^{loc}, \cdot)$ is such that $\fRHom_{\cA_X}(\cA_X^{loc}, \Der^{-}(\cA_X))\subset  \Der^{-}(\cA_X)$.
\end{cor}

\end{document}